\documentclass[12pt]{amsart}
\usepackage{amssymb,amsmath,amsthm,comment}
\usepackage{pdfpages,graphicx} 
\usepackage[section]{placeins} 
\usepackage{wrapfig}
\usepackage{float}
\usepackage{lineno}

\usepackage{setspace}
\newcommand{\shrinkmargins}[1]{
  \addtolength{\textheight}{#1\topmargin}
  \addtolength{\textheight}{#1\topmargin}
  \addtolength{\textwidth}{#1\oddsidemargin}
  \addtolength{\textwidth}{#1\evensidemargin}
  \addtolength{\topmargin}{-#1\topmargin}
  \addtolength{\oddsidemargin}{-#1\oddsidemargin}
  \addtolength{\evensidemargin}{-#1\evensidemargin}
  }
\shrinkmargins{0.5}

\newtheorem{theorem}{Theorem}

\newtheorem*{theorem*}{Theorem}

\newtheorem{conjecture}{Conjecture}
\newtheorem{proposition}[theorem]{Proposition}
{Claim}

\theoremstyle{remark}

\numberwithin{theorem}{section} \numberwithin{equation}{section}

\usepackage{caption}
\usepackage{multirow}

\def\func#1{\mathop{\rm #1}}%

\begin{document}
\title[BO by Induction]{Proof of the Bessenrodt--Ono inequality by Induction}
\author{Bernhard Heim }
\address{Lehrstuhl A f\"{u}r Mathematik, RWTH Aachen University, 52056 Aachen, Germany}
\email{bernhard.heim@rwth-aachen.de}
\author{Markus Neuhauser}
\address{Kutaisi International University, 5/7, Youth Avenue,  Kutaisi, 4600 Georgia}
\email{markus.neuhauser@kiu.edu.ge}
\begin{abstract}
In 2016 Bessenrodt--Ono discovered an inequality addressing
additive and multiplicative properties of the partition function.
Generalization by several authors have been given; on partitions with rank in a given residue class by
Hou--Jagadeesan and Males, on $k$-regular partitions by Beckwith--Bessenrodt, on $k$-colored partitions by
Chern, Fu, Tang, and Heim--Neuhauser on their polynomization, and Dawsey--Masri on the Andrews ${\it spt}$-function. 
The proofs depend on non-trivial asymptotic formulas related to the circle method on one side, or 
a sophisticated combinatorial proof invented by Alanazi--Gagola--Munagi. 
We offer in this paper a new proof of
the Bessenrodt--Ono inequality, which
is built on a
well-known recursion formula for partition numbers. We extend the proof to the
result of Chern--Fu--Tang and its polynomization. Finally, we also obtain a new result.
\end{abstract}
\maketitle
\section{Introduction and main results}
Bessenrodt and Ono \cite{BO16} discovered
an  interesting 
inequality for partition numbers $p(n)$. Their discovery nudged further research and new results
in several directions.

A partition $\lambda$
of $n$ is any non-increasing sequence $\lambda_1,
\ldots,  \lambda_d$ of positive integers whose sum is $n$.
The number of partitions of $n$ is denoted by $p(n)$. Table \ref{p1} records the first 10 values and the special case
$p\left( 0\right) =1$.
\begin{table}[H]
\[
\begin{array}{cccccccccccc}
\hline
n & 0 & 1 & 2 & 3 & 4 & 5 & 6 & 7 & 8 & 9 & 10 \\ \hline \hline
p\left( n \right) & 1 & 1 & 2 &
3 &
5 &
7 &
11 &
15 &
22 & 30 & 42 \\
\hline
\end{array}
\]
\caption{\label{p1}
Values of $ p \left( n \right)$ for $ 0 \leq n \leq 10$.}
\end{table}
A partition is called a $k$-colored partition of $n$
if each part can appear in $k$ colors.
\begin{theorem}[Bessenrodt--Ono 2016]
Let $a$ and $b$ be positive integers. Let $a,b \geq 2$ and $a+b \geq 10$. Then
\begin{equation}
p(a) \, p(b) > p(a+b).
\end{equation}
\end{theorem}
The proof depends on results of Rademacher \cite{Ra37} and Lehmer \cite{Le39} on the size of the partition numbers, 
built on the circle method of Hardy and Ramanujan.
The partition numbers are considered as the coefficients of a weakly modular form, 
the reciprocal of the Dedekind $\eta$-function \cite{On03}. 
The modularity condition is very strong, and it would be desirable to also have other proofs.

Alanazi, Gagola, and Munagi \cite{AGM17} came up with an impressive and subtle combinatorial proof.
Chern, Fu, and Tang \cite{CFT18} generalized Bessenrodt and 
Ono's theorem to $k$-colored partitions.
 
\begin{theorem}[Chern, Fu, Tang 2018] \label{Fu} \ \\
Let $a,b,k$ be natural numbers. Let $p_{-k}(n)$ denote the number of $k$-colored partitions of $n$.
Let $k>1$. Then 
\begin{equation} \label{k=2}
p_{-k}(a) \,  p_{-k}(b)    >     p_{-k}(a+b), 
\end{equation}
except for $(a,b,k) \in \left\{ (1,1,2), (1,2,2), (2,1,2), (1,3,2), (3,1,2), (1,1,3)\right\} $.
\end{theorem}
Two proofs are given. One depends on the Bessenrodt--Ono inequality and some new ideas, the second
is combinatorial and motived by the work of Alanazi, Gagola, and Munagi \cite{AGM17}.

Bessenrodt--Ono type inequalities appeared also in works by Beckwith and Bessenrodt \cite{BB16} on $k$-regular partitions
and Hou and Jagadeesan \cite{HJ18} on the numbers of partitions with ranks in a given residue class modulo $3$.
Males \cite{Ma20} obtained results for general $t$ and Dawsey and Masri \cite{DM19} obtained new results for the Andrews {\it{spt}}-function.
The authors of this paper generalized the Chern--Fu--Tang Theorem to D'Arcais polynomials,
also known as Nekrasov--Okounkov polynomials
\cite{Ne55,NO06,Ha10,CFT18, HN20,HNT20}.
Let
\begin{equation}
P_n(x):= \frac{x}{n} \sum_{k=1}^n \sigma(k) \, P_{n-k}(x),
\label{eq:xrecurrence}
\end{equation}
with $\sigma(k) := \sum_{d \mid k} d$ and the initial condition $
P_{0}\left( x\right) =1$. Then $p_{-k}(n)= P_n(k)$ and $p(n)=p_{-1}(n)= P_n(1)$.
\begin{theorem}[Heim, Neuhauser 2020] \label{Hauptsatz}
Let $a,b \in \mathbb{N}$, $a+b>2$, and $x >2$. Then
\begin{equation}
P_a(x) \, P_b(x)  >  P_{a+b}(x).  \label{conjecture}
\end{equation}
The case $x=2$ is true for $a+b>4$.
\end{theorem}
The proof is based on the result for $2$-colored partitions \cite{CFT18}, Lehmer's
\cite{Le39} lower and upper bound on the partition numbers, and a detailed analysis of the
growth of the derivative of $P_{a,b}(x)$.

These results are related to the work of 
Griffin, Ono, Rolen, and Zagier \cite{GORZ19} on Jensen polynomials and their hyperbolicity.
This includes work of Nicolas \cite{Ni78}
and De Salvo and Pak \cite{DP15} on the log-concavity of the partition function $p(n)$ for $n >25$,
and results and a conjecture of Chen, Jia, and Wang \cite{CJW19} for the higher order Tur\'{a}n inequalities.
Related to their work on the Bessenrodt--Ono inequality, Chern, Fu, and Tang, came up with a
subtle and explicit conjecture on $k$-colored partitions. The positivity of the
discriminant in the case of degree $2$ is equal to the log-concavity of the considered sequences
coded in the Jensen polynomials.
\begin{conjecture} [Chern, Fu, Tang 2018] \label{CFT} 
Let $n > m \geq 1$ and $k \geq 2$ then, except for $(k,n,m) = (2,6,4)$,
\begin{equation}
p_{-k} (n-1) \, p_{-k} (m+1)  \geq \, p_{-k} (n) \, p_{-k} (m).
\end{equation}
\end{conjecture}
The conjecture was extended to D'Arcais polynomials.
\begin{conjecture}[Heim, Neuhauser 2020] \label{HN: Conjecture 2}
Let $a > b \geq 0$ be integers. Then for all $x \geq 2$:
\begin{equation}\label{co}
\Delta_{a,b}(x) := P_{a-1}(x) P_{b+1}(x) - P_{a}(x) P_{b}(x) \geq 0,
\end{equation}
except for $b=0$ and $(a,b) = (6,4)$. The inequality (\ref{co}) is still true for $x \geq 3$ for $b=0$ and
for $x \geq x_{6,4}$ for $(a,b)=(6,4)$. Here $x_{a,b}$ is the largest real root of $\Delta_{a,b}(x)$.
\end{conjecture} 
Based on a recently obtained exact formula of Rademacher type (based on the circle method) for $P_{n}(x)$ with $x>0$ and $n > \frac{x}{24}$ obtained by
Iskander, Jain, and Talvola \cite{IJT20}, new strong estimates on the Bessel function, the determination of the
main term of $P_{a,b}(x)$, and some sophisticated computer calculation,
Bringmann, Kane, Rolen, and Tripp \cite{BKRT20} were able to proof the
conjecture of Chern, Fu, and Tang. They essentially proved the conjecture for $x=2,3,4$ and applied a result of Hoggar on
the convolution of log-concave sequences. They further proved that the conjecture of Heim and Neuhauser is true (\cite{BKRT20}, Corollary 1.3) for all pairs $(a,b)$, where 
\begin{equation}
b \geq \max\, \left\{ 2 \, x^{11} + \frac{x}{24}, \, \frac{100}{x - 24} + \frac{x}{24} \right\}.
\end{equation}
The crux of these methods is that one needs for the general case Rademacher type formulas, and has to check
the conjecture for each $x$ for finitely many cases.
In the discrete case for the Bessenrodt--Ono inequality there is a  combinatorial proof available,
this was also requested in \cite{BKRT20}, see concluding remarks (5), for the Chern--Fu--Tang conjecture.

In this paper we offer a new proof for the Bessenrodt--Ono inequality for partition numbers.
Ingredients are the well-known recurrence property:
\begin{equation}\label{recurrence}
n \, p(n) = \sum_{k=1}^n \sigma(k) \, p(n-k),
\end{equation}
and an elementary upper bound of $\sigma(n)$ and lower bound of $p(n)$.
The proof also perfectly fits to the Bessenrodt--Ono inequality for $k$-colored partitions $p_{-k}(n)$ and 
D'Arcais polynomials $P_n(k)$. With slight modification and including some extra considerations, we
obtain new proofs of Theorem \ref{Fu} of Chern--Fu--Tang and Theorem \ref{Hauptsatz}.
Finally, we obtain the following Theorem, to give evidence that the proof method offered in this
paper also gives an extension of Theorem \ref{Hauptsatz}.

\begin{theorem}\label{new}
Let $a,b \in \mathbb{N}$, $a+b> 4$, and $x > 1.8$. Then
\begin{equation}
P_a(x) \, P_b(x)  >  P_{a+b}(x). 
\end{equation}
\end{theorem}

It is hoped that the results of this paper lead to a new proof of the former Conjecture \ref{Fu}
of Chern--Fu--Tang \cite{BKRT20} and
to a proof of the Conjecture \ref{HN: Conjecture 2}.

\section{New Proof of the Bessenrodt--Ono inequality}\label{proof:BO}
We estimate the divisor sum function $\sigma(n)$ and the partition numbers $p(n)$
with the following upper and lower bounds:
\begin{eqnarray}
\sigma(n) & \leq & n \,  \big( 1 + \func{ln}(n) \big) \label{sigma} \\
p(n) & \geq & \sum_{k=1}^m \frac{1}{k!}  \binom{n-1}{k
-1} \text{ for all } m \in \mathbb{N}.\label{easy}
\end{eqnarray}
The upper bound for $\sigma(n)$ follows easily by integral comparison.
We have a strict upper bound for $n>1$. 
There are exactly $\binom{n-1}{
k-1}$
ways to represent $n$ as a sum of exactly $
k$ positive integers. 
Thus, we have at most $k!$ compositions representing the same partition.
For a generalization we refer to Section~\ref{application} and the relation to
associated Laguerre polynomials. 
\begin{proof}
Let $A:=2$ and $B:=10$. Let $n \geq B$. We say the statement $S(n)$ is true if for all partitions
$n=a +b$ with $a,b \geq A$:
\begin{equation}\label{pab}
p_{a,b}(n):= p(a) \, p(b) - p(a+b)>0.
\end{equation}

By symmetry the claim can be reduced to all pairs $(a,b)$ with $A \leq b \leq a$. We assume that
$n >N_0>1$ and $S(m)$ are true for all $ B \leq m \leq n-1$. For $ B \leq \ell \leq N_0$
we show $S(\ell)$ by a direct computer calculation with PARI/GP.

Note, it is sufficient to prove $S(n)$ for
fixed $A \leq b \leq a$ with $a+b=n$. 
We have introduced the constants $A$, $B$ and $N_0$ to make the generalization of the
given proof in our applications
transparent. It will turn out that we can chose $N_{0}=2184$.

We utilize the recurrence (\ref{recurrence})
and obtain for $p_{a,b}\left(n \right) =L+R$ the
expressions:
\begin{eqnarray}\label{startl}
L
&:=& -
\sum_{k=1}^{b}
\frac{\sigma \left( k+a \right)}{a+b}
p\left(b-k\right) \\
R
&:=&
\sum _{k=1}^{a}
\Big(
\frac{\sigma(k)}{a}\,
p(a-k) \, 
p(b) -\frac{\sigma(k)}{a+b} \,
p(a+b-k)\Big) .
\label{startr}
\end{eqnarray}
We show that $p_{a,b}\left(n \right)>0$. Further, we will refine the right sum $R$ into 
$$R=R_1+R_2+R_3.$$
\subsection{Left sum $L$}
We have
\begin{equation}
L = -\sum_{k=1}^{b}
\frac{\sigma \left( k+a \right)}{a+b}
p\left(b-k\right) > -b \, p\left( b \right) \, 
\big( 1 + \ln (a+b)\big)  .
\end{equation}
\subsection{Right sum $R
$}
The dominant term is related to $k=1$ appearing in the right sum $R$.
Note that the induction 
hypothesis cannot be applied in general to all terms. 
Therefore, we decompose the right sum $R$
into three parts. Let 
\begin{equation}
k_0:= a-\max \left\{ B -b,  A \right\} + 1. 
\end{equation}
Thus, $k_0= a-\max \left\{ B -b-1,  A-1 \right\}= a-\max \left\{ 9-b, 1 \right\}$ for $A=2$ and $B=10$. Let
\begin{equation}\label{decomposition}
R_{1}:=\sum_{k=1}^{1}  \,f_k(a,b)       , \,\, 
R_{2}:=\sum_{k=2}^{k_0-1} \,f_k(a,b) , \,\, 
R_{3}:=\sum_{k=k_0}^{a} \,f_k(a,b) .
\end{equation}
\begin{equation}\label{fkab}
f_k(a,b):=
\frac{\sigma(k)}{a}\,
p(a-k) \, 
p(b) -\frac{\sigma(k)}{a+b} \,
p(a+b-k) .
\end{equation}
\subsubsection{The sum $R_{1}$} \ \newline
The first sum related to $k=1$ is simplified by the induction hypothesis: 
\begin{equation*}
- p(a+b-1) > - p(a-1) \, p(b).
\end{equation*}
Thus, we obtain
the lower bound:
\begin{equation}
R_1 =
\frac{p(a-1) \, p(b)}{a} - \frac{p(a+b-1)}{a+b}>
 \frac{b}{2a^{2}}\, p\left( a-1\right) p\left( b\right) .
\end{equation}
\subsubsection{The sum $R_{2
}$}\ \newline
The second sum, using again the induction hypothesis, can be 
estimated from
below with $0$. This will be sufficient for our purpose:
$R_2 > 0$.
\subsubsection{The sum $R_{3}$} \ \newline
We split the third sum again into three parts: $R_3= R_{
31}+R_{
32}+R_{
33}$, where
\begin{eqnarray}
R_{31}:=
\sum_{k=k_0}^{a-A}   \, f_k(a,b)          & > &
\sum_{k=k_0}^{a-A-1}  
\sigma(k) \left(   \frac{p(2)\, p(b)}{a} - \frac{p(a+b-k)}{a+b}\right) \label{big}
\\
& > &  \nonumber
5 \, \big( 4 - p(8) \big) \, \big(1+\ln \left( a\right) \big)  \label{a-1},\\
R_{32}:=\sum
_{k=a-A+1}^{a-1}    \, f_k(a,b)                       & = & 
\sigma(a-1) \left( \frac{p(1) \, p(b)}{a} - \frac{p(b+1)}{a+b}\right)         \\
& > & -\frac{\sigma(a-1) p(b)}{a} > - p\left( b\right)  \, \big( 1 + \ln \left( a\right) \big) ,\nonumber
\\
R_{33}:= \sum_{k=a}^{a}  \, f_k(a,b)  & \geq & 0 .
\end{eqnarray}
We first use that
$p\left( 9-b\right) p\left( b\right) \geq p\left( 9\right) $.
This leads to at most $5$ summands.
The next estimation in (\ref{big}) follows from $p(a-k) \geq p(2)$. 
For the second estimation we use that
\begin{equation*}
p\left( b \right) \geq p \left( 2 \right) = 2, \,\,
\frac{1}{a+b}<\frac{1}{a}, \text{ and } p(a+b-k) \leq p(8).
\end{equation*}
We also refer to (\ref{sigma}).
The estimation in (\ref{a-1}) follows from the obvious inequality $\frac{1}{a+b} < \frac{1}{a}$ and
$ p(b+1) < 2 \, p(b)$.
Thus,
\begin{equation}
R_3 >
-\left( 1+\ln \left( a-1\right) \right) 
\frac{
\left( a-1\right) \, p(b)}{a+b}
-90 \left( 1+\ln \left( a\right) \right) .
\end{equation}
Since $b\geq 2$, $90 \leq  \frac{45}{2
}\, b \, p\left( b\right) $,
$\frac{1}{a+b}\geq
\frac{1}{2a}$, and
$a-1<
a<
a+b$
we obtain
\begin{equation}
R_3 > 
-b \, p\left( b\right) \left(
1+\ln \left( a+b \right) \right) \frac{1+45}{2}.
\end{equation}
\subsection{Final step}\label{final}
Putting everything together leads to
\begin{eqnarray}
p_{a,b} \left( n \right) &> &  
\frac{ b \, p \left( b \right)}{ 2 \, a^2}   
\left( -48 \, a^2 \, \left( 1+\ln \left( a+b\right) \right) +
p\left(a-1)\right)\right)\\
&>&  \frac{ b \, p \left( b \right)}{ 2 \, a^2}  
\left( -48 \, a^2 \, \left( 1+\ln \left( 2a\right) \right) +  \sum_{\ell=1}^5 \frac{1}{\ell !}  \binom{a-2}{\ell -1} \right). \label{final}
\end{eqnarray}
In the last step we used the property (\ref{easy}). For $a\rightarrow \infty $
we can immediately observe that this is positive
since the sum is a polynomial of
degree~$4$ in~$a$ which grows faster than $a^{2}\left( 1+\ln \left( 2a\right) \right) $.
In fact the expression (\ref{final}) is positive for all $ a \geq 1093
$. Note that if $a<1093$ then $a+b\leq 2a\leq 2184=N_{0}$.
Therefore,
we have shown that $p_{a,b}\left(n \right) >0$, which proves the Theorem.
\end{proof}

\section{Applications}\label{application}
We extend the proof method presented in Section \ref{proof:BO}
to prove the Bessenrodt--Ono inequality
for $k$-colored partitions
and its extension to D'Arcais polynomials.
Let $k \in \mathbb{N}$. Then $P_n(k)$ is equal to the $k$-colored partition number and $p(n)= P_n(1)$. We define
\begin{equation}
P_{a,b}\left( x
\right) := P_a(x) P_b(x) -  P_{a+b}(x).
\end{equation}
Before we start, we fix the following lower bound for the D'Arcais polynomials $P_n(x)$.
Let $x$ and $\alpha$ be real numbers with $x \geq 0$ and $\alpha >-1$. Let
$L_{n}^{\left( \alpha \right) }\left( x\right) $ be the $\alpha$-associated 
Laguerre polynomial. Then $P_n(x) \geq \frac{x}{n} L_{n-1}^{(1)}(-x)$. We refer to \cite{HLN19}.
This implies
\begin{equation}
P_n(x) \geq 
\sum_{k=1}^{m}
\binom{n-1}{k-1}\frac{x^{k}}{k!} \text{ for all } m \in \mathbb{N}.
\label{xeasy}
\end{equation}

\subsection{Bessenrodt--Ono for $x>3$ and arbitrary $a$ and $b$}
We first prove that $P_{a,b}(x)>3$ is true for all $a$ and $b \in \mathbb{N}$.
Since there are no
restrictions on $a$ and $b$, the proof will be straightforward.
\begin{proposition}
Let $a$ and $b$ be positive integers with $a,b \geq 1$ and $x$ a real number with $ x >3$.
Then $P_{a,b}(x)>0$.
\end{proposition}
\begin{proof}
We follow the proof by induction of the Bessenrodt--Ono inequality presented in Section \ref{proof:BO}.
Let $n \geq 2$. The statement $S\left( n\right) $
is true if for all partitions
$n=a +b$ with $a,b \geq 1$ holds $P_{a,b}(x)
>0$ for $x >3$.
Let $1\leq b \leq a$.
Let $n >N_0$ and $S(m)$ be true for all $ 2
\leq m \leq n-1$. We show that it is sufficient to put $N_0=
14$.
Note that $S(n)$ is true for all $1 \leq a,b \leq 14$ (see
Table~\ref{nullstellen}).
Let $P_{a,b}(x) = L + R$ with $L$ and $R$ defined as in (\ref{startl}) and (\ref{startr}),
where we have to substitute $p(n)$ by $P_n(x)$. 

\subsubsection{Left sum $L$}
The left sum satisfies
\begin{equation*} L > -b \, P_{b}\left( x
\right) \, 
\left( 1 + \ln \left( a+b\right) \right)  .
\end{equation*}

\subsubsection{Right sum $R$}
We take care about the dominating term for $k=1$.
Similar to (\ref{decomposition}) with 
\begin{equation*}
f_k(a,b):=
\frac{\sigma(k)}{a}\,
P_{a-k}(x) \, 
P_b(x) -\frac{\sigma(k)}{a+b} \,
P_{a+b-k}\left( x\right)
\end{equation*}
we study $R=R_1+R_2+R_3$. By the induction hypothesis we get $R_2 >0$.
And since we have no extra condition on $a$ and $b$ we also get $R_3 \geq 0$.
Thus, only $R_1$ attached to $k=1$ contributes and leads to
\begin{equation*}
R >  \frac{b}{2 \, a^2} \, P_{a-1}(x) \, P_{b}(x).
\end{equation*}
\subsubsection{Final step}
Putting everything together leads to
\begin{eqnarray}
P_{a,b} \left(
x\right) &
>&
\frac{ b \, P_{b}\left( x\right) }{2a^{2}}\left(
-2a^{2}\left( 1+\ln \left( a+b\right) \right) +P_{a-1}\left(
x\right) \right) \label{eq:nonumber} \\
&>&
\frac{ b \, P_{b} \left(
x\right)}{ 2 \, a^2}  
\left( -2a^{2}\left( 1+\ln \left( 2a\right) \right) +
\sum_{\ell=1}^{6
}
\binom{a-2}{\ell -1}\frac{
3^{\ell }}{\ell !}\right)
. \label{xfinal}
\end{eqnarray}
In the last step we used the property (\ref{xeasy}) and that $x>3$.
We obtain that the expression (\ref{xfinal}) is positive for all $ a \geq 12
$. Now $200
\left( 1+\ln \left( 20
\right) \right) <800
$ and $P_{8}\left( 3\right) =810$. Therefore,
(\ref{eq:nonumber}) is positive already for $a\geq 8
$. Since the leading
coefficient $\frac{1}{a!b!}-\frac{1}{\left( a+b\right) !}$ of
$P_{a,b}\left( x\right) $ is positive, we only have to check the largest real
zero
of all remaining $P_{a,b}\left( x\right)$. This was done for $2
\leq b+
a\leq N_{0}
=14$ with
PARI/GP (Table~\ref{nullstellen}). Note that in the case of
$P_{1,1}\left( x\right) =\left( x-3\right) \frac{x}{2}$ the largest real zero
is exactly $3$.
\end{proof}
\subsection{The $2$-colored partitions}
Chern--Fu--Tang (\cite{CFT18}, Theorem 1.2) proved the following result. 
\begin{theorem}
Let $a$ and $b$ be positive integers with $a,b \geq 1$ and $n=a+b \geq 5$.
Then $P_{a,b}(2)>0$.
\end{theorem}
\begin{proof}
We have $A=1$, $B=5$, and $k_0= a - \max \left\{ 5-b,1\right\}+1 $.
Let $n\geq 5$ and $S(n)$ be the statement: $P_{a,b}(n)>0$ for all $a,b \geq 1$ with $n=a+b$.
A numerical calculation with PARI/GP shows that $S(m)$ is true for all $5 \leq m \leq N_0=28$.
We prove $S(n)$ by induction on $n$. Let $n=a +b >N_{0}$ and $1 \leq b \leq a$. Let $P_{a,b}(2)= L +R$.
Then $$L > -b \, P_b(2) \, \big( 1 + \ln(a+b) \big).$$ Further, 
\begin{equation*}
 R_1 > \frac{b }{ 2 \, a^2} \, P_{a-1}(2) \, P_b(2).
\end{equation*}
We have $R_2 \geq 0$ by the induction hypothesis and $R_3 \geq 0$ for $b >3$.
Moreover
$$\frac{P_{a-k}\left( 2\right) P_{b}\left( 2\right) }{a}-\frac{P_{a+b-k}\left( 2\right) }{a+b}>\frac{P_{a-k}\left( 2\right) P_{b}\left( 2\right) -P_{a+b-k}\left( 2\right) }{a}.$$
From the induction hypothesis and
Table~\ref{zwei} we see that this is 
non-negative for $a\geq 2$ or $b\geq 2$. For $a=1=b$ we obtain
$\left( P_{1}\left( 2\right) \right) ^{2}-P_{2
}\left( 2\right) =-1$.
This leads to
\begin{eqnarray}
P_{a,b} \left(
2\right) &> &
\frac{ b \,
P_{b}\left(
2\right)}{ 2 \, a^2}
\left( -3
\, a^2 \, \left( 1+\ln \left( a+b\right) \right) +
P_{a-1}\left(
2\right) \right)\\
&>&  \frac{ b \,
P_{b}\left( 2 \right)}{ 2 \, a^2}
\left( -3
\, a^2 \, \left( 1+\ln \left( 2a\right) \right) +  \sum_{\ell=1}^{5} \frac{2^{\ell }}{\ell !}  \binom{a-2}{\ell -1} \right). \label{eq:final2}
\end{eqnarray}
In the last step we used the property (\ref{xeasy}). For $a\rightarrow \infty $
we can immediately observe that this is positive
since the sum is a polynomial of
degree~$4$ in~$a$, which grows faster than $a^{2}\left( 1+\ln \left( 2a\right) \right) $.
In fact, the expression (\ref{eq:final2}) is positive for all $ a \geq 15
$.
For the remaining $5\leq a+b\leq 28$
we have checked with PARI/GP that $P_{a,b}\left( 2\right) >0$.
\end{proof}

\begin{table}[H]
\[
\begin{array}{ccccc}
\hline
a\backslash b & 1 & 2 & 3 & 4 \\ \hline \hline
1 & -1 & 0 & 0 & 4 \\
2 & 0 & 5 & 14 & 35 \\
3 & 0 & 14 & 35 & 90 \\
4 & 4 & 35 & 90 & 215\\ \hline
\end{array}
\]
\caption{\label{zwei}Values of $P_{a,b}\left( 2
\right) $ for $a,b\in \left\{ 1,2,3,4\right\} $.}
\end{table}

\subsection{Proof of Theorem \ref{new}}

Let $n \geq 5
$. The statement $S\left( n\right) $
is true if for all partitions
$n=a +b$ with $a,b \geq 1$
holds $P_{a,b}(x)
>0$ for $x >1.8$.
Let $1\leq b \leq a$.
Let $n >N_{0}$ and $S(m)$ be true for all $ 5
\leq m \leq n-1$. We show that it is sufficient to put $N_{0}=28
$.
Note that $S(n)$ is true for all $1 \leq a,b \leq 28
$ (compare
Table~\ref{nullstellen}).
Let $P_{a,b}(x) = L + R$ with $L$ and $R$ defined as in (\ref{startl}) and (\ref{startr}),
where we have to substitute $p(n)$ by $P_n(x)$. 
\subsubsection{Left sum $L$}
The left sum satisfies
\begin{equation*} L > -b \, P_{b}\left( x
\right) \, 
\left( 1 + \ln \left( a+b\right) \right)  .
\end{equation*}

\subsubsection{Right sum $R$}
We study $R=R_1+R_2+R_3$. By the induction hypothesis we get $R_2 >0$.
$R_1$ associated to $k=1$ leads to
\begin{equation*}
R_{1} >  \frac{b}{2 \, a^2} \, P_{a-1}(x) \, P_{b}(x).
\end{equation*}
We find that $R_{32}=0$ and $R_{33}\geq 0$. For $b\geq 4$ we
can apply the induction hypothesis and obtain $R_{31}>0$.
In case $1\leq b\leq 3$ we find from Table~\ref{nullstellen} that
$P_{a-k,b}\left( 1.8\right) >0$ for
$\left( a-k,b\right) \notin \left\{ \left( 2,1\right) ,\left( 3,1\right) \right\} $.
Therefore,
$$R_{3,1}>\frac{\sigma \left( a-3\right) P_{3,1}\left( x\right) +\sigma \left( a-2\right) P_{2,1}\left( x\right)
}{a}.$$ It can be checked that the polynomials
$P_{2,1}\left( x\right) =\frac{1}{3}\*x^3
 - \frac{4}{3}\*x$ and $P_{3,1}\left( x\right) =\frac{1}{8}\*x^4
 + \frac{3}{4}\*x^3
 - \frac{9}{8}\*x^2
 - \frac{7}{4}\*x$ are monotonically increasing for
$x\geq 1.8$. Thus,
$$ R > R_{3,1}>\frac{
\sigma \left( a-3\right) P_{3,1}\left( 1.8\right) +\sigma \left( a-2\right) P_{2,1}\left( 1.8\right) }{a}
>-1.5648\left( 1+\ln \left( a\right) \right). $$

\subsubsection{Final step}
Putting everything together leads to
\begin{eqnarray}
P_{a,b} \left(
x\right) &
>&
\frac{ b \, P_{b}\left( x\right) }{2a^{2}}\left(
-3.8a^{2}\left( 1+\ln \left( a+b\right) \right) +P_{a-1}\left(
x\right) \right) \label{eq:1.8nonumber} \\
&>&
\frac{ b \, P_{b} \left(
x\right)}{ 2 \, a^2}  
\left( -3.8a^{2}\left( 1+\ln \left( 2a\right) \right) +
\sum_{\ell=1}^{9
}
\binom{a-2}{\ell -1}\frac{
1.8
^{\ell }}{\ell !}\right)
. \label{x1.8final}
\end{eqnarray}
In the last step we used the property (\ref{xeasy}) and that $x
\geq 1.8
$.
We obtain that the expression (\ref{x1.8final}) is positive for all $ a \geq 30
$. Now $450
\left( 1+\ln \left( 30
\right) \right) <2000<
P_{15}\left( 1.8\right) $. Therefore,
(\ref{eq:1.8nonumber}) is already positive for $a\geq 15
$. Since the leading
coefficient 
of
$P_{a,b}\left( x\right) $ is positive we only have to determine the largest real
zero
of all remaining $P_{a,b}\left( x\right)$. This was done for $1\leq b
,a\leq 28$ with
PARI/GP (compare Table \ref{nullstellen}).

\begin{table}[H]
\[
\begin{array}{ccccccccccccccc}
\hline
a\backslash b&1&2&3&4&5&6&7&8&9&10&11&12&13&14\\ \hline \hline
1&3.0&2.0&2.0&1.7&1.7&1.6&1.6&1.5&1.5&1.4&1.5&1.4&1.4&1.4\\
2&2.0&1.4&1.2&1.1&1.1&1.0&1.0&0.9&0.9&0.9&0.9&0.9&0.9&0.8\\
3&2.0&1.2&1.2&1.0&1.0&0.9&0.9&0.8&0.9&0.8&0.8&0.8&0.8&0.8\\
4&1.7&1.1&1.0&0.9&0.9&0.8&0.8&0.7&0.7&0.7&0.7&0.6&0.6&0.6\\
5&1.7&1.1&1.0&0.9&0.9&0.7&0.8&0.7&0.7&0.7&0.7&0.6&0.7&0.6\\
6&1.6&1.0&0.9&0.8&0.7&0.7&0.7&0.6&0.6&0.6&0.6&0.5&0.5&0.5\\
7&1.6&1.0&0.9&0.8&0.8&0.7&0.7&0.6&0.6&0.6&0.6&0.5&0.6&0.5\\
8&1.5&0.9&0.8&0.7&0.7&0.6&0.6&0.6&0.6&0.5&0.5&0.5&0.5&0.5\\
9&1.5&0.9&0.9&0.7&0.7&0.6&0.6&0.6&0.6&0.5&0.5&0.5&0.5&0.5\\
10&1.4&0.9&0.8&0.7&0.7&0.6&0.6&0.5&0.5&0.5&0.5&0.5&0.5&0.4\\
11&1.5&0.9&0.8&0.7&0.7&0.6&0.6&0.5&0.5&0.5&0.5&0.5&0.5&0.4\\
12&1.4&0.9&0.8&0.6&0.6&0.5&0.5&0.5&0.5&0.5&0.5&0.4&0.4&0.4\\
13&1.4&0.9&0.8&0.6&0.7&0.5&0.6&0.5&0.5&0.5&0.5&0.4&0.4&0.4\\
14&1.4&0.8&0.8&0.6&0.6&0.5&0.5&0.5&0.5&0.4&0.4&0.4&0.4&0.4\\ \hline
\end{array}
\]
\caption{\label{nullstellen}
Approximative largest real zeros of $P_{a,b}\left( x\right) $ for $1\leq a,b\leq
14$.}
\end{table}


\end{document}